\documentclass[reqno,11pt,oneside]{amsart}
\usepackage[colorlinks, citecolor=magenta,linkcolor=blue]{hyperref}
\usepackage{amssymb}

\usepackage{graphicx,subfigure,epstopdf}
\usepackage{geometry}  
\usepackage[english]{babel}
\usepackage{epsfig}
\usepackage{float}
\usepackage{todo}
\usepackage{color}
\usepackage{physics}
\usepackage{tensor} 
\usepackage{amsmath,amsthm,amsfonts,amsfonts}
\usepackage{enumerate}
\usepackage{array}
\theoremstyle{plain}
\newtheorem{theorem}{Theorem}[section]

\newcommand{\R}{\mathbb{R}}

\newcommand{\dive}{\operatorname{div}}

\theoremstyle{definition}
\newtheorem{definition}[theorem]{Definition}
\newtheorem{proposition}[theorem]{Proposition}
\newtheorem{remark}[theorem]{Remark}
\newtheorem{corollary}[theorem]{Corollary}
\newtheorem{lemma}[theorem]{Lemma}
\newtheorem{example}[theorem]{Example}

\renewenvironment{proof}{{\noindent \bf  Proof.}}{\qed}

\def\LHM{\mathcal H^\nu}
\def\LHMp{ H^\nu}
\def\LH{\mathcal H^\nu_0}
\def\LHp{ H^\nu_0}
\def\LHkill{\mathcal H^{\nu,\text{kill}}_0}

\makeatletter
\let\orgdescriptionlabel\descriptionlabel
\renewcommand*{\descriptionlabel}[1]{%
  \let\orglabel\label
  \let\label\@gobble
  \phantomsection
  \edef\@currentlabel{#1}%
  \let\label\orglabel
  \orgdescriptionlabel{#1}%
}
\makeatother

\allowdisplaybreaks

\numberwithin{equation}{section}
\begin{document}
\title[Space-time coupled evolution equations and their stochastic solutions]{Space-time coupled evolution equations and their stochastic solutions}

\author[John Hermanz]{John Herman}
\address{John Herman \newline Department of Mathematics, University of Warwick, UK}
\email {j.a.Herman@warwick.ac.uk}

\author[Ifan Johnston]{$\,\,$Ifan Johnston$\,$}
\address{Ifan Johnston \newline Department of Mathematics, University of Warwick, UK}
\email {i.johnston@warwick.ac.uk}

\author[Lorenzo Toniazzi]{$\,\,$Lorenzo Toniazzi$\,$}
\address{Lorenzo Toniazzi \newline Department of Applied Physics and Applied Mathematics, Columbia University NY, USA}
\email {lt2739@columbia.edu}
\date{\today}

\begin{abstract}
We consider a class of space-time coupled evolution equations (CEEs), obtained by a subordination of the heat operator. Our CEEs reformulate and extend known governing equations of non-Markovian processes  arising as scaling limits of continuous time random walks, with widespread applications. In particular we allow for initial conditions imposed on the past, general spatial operators on Euclidean domains and a forcing term.  We prove existence, uniqueness and stochastic representation for solutions. 
\end{abstract}

\maketitle

\section{Introduction}
We study the space-time coupled evolution equation (CEE)
	\begin{equation}\label{preRLHO_nl}
\left\{
\begin{split}
 H^\nu u(t,x)    &=-f(t,x), & \text{in }& (0,T]\times \Omega,\\ 
  u(t,x)&=\phi(t,x),  &\text{in }&(-\infty,0]\times \Omega,
\end{split}
\right.
\end{equation}
where $f$ and $\phi$ are given data and 
	\begin{equation}
H^\nu u(t,x)= \int_0^\infty \left(e^{r\mathcal L}u(t-r,x)-u(t,x) \right)\nu(r)\,dr,\quad t>0,
	\label{eq:Hnu}
	\end{equation}
	so that $-H^\nu=(\partial_t-\mathcal L)^\nu$ is the subordination of the heat operator $(\partial_t-\mathcal L)$ by an infinite L\'evy measure $\nu$. 	Here the Markovian semigroup $\{e^{r\mathcal L}\}_{r\ge0}$ acts on the space variable $\Omega\subset\mathbb R^d$, and we denote the associated process by $B=r\mapsto B_r$. As our main result,  we prove the stochastic representation for the solution to \eqref{preRLHO_nl} to be
	\begin{equation}
u(t,x)=\mathbf E\left[\phi\left(t-S^\nu_{\tau_0(t)},B^x_{S^\nu_{\tau_0(t)}}\right)\mathbf1_{\{\tau_0(t)<\tau_0(\tau_\Omega(x))\}} \right] + \mathbf E\left[\int_0^{\tau_0(t)\wedge \tau_0(\tau_\Omega(x))}f\left(t-S^\nu_{s},B^x_{S^\nu_s}\right)ds \right],
\label{eq:SR}
\end{equation}
where $S^\nu$ is the L\'evy subordinator induced by $\nu$, $S^\nu$ is independent of $B^x$,   with $x$ denoting the starting point of $B$, $\tau_0(t)=\inf\{r>0: t-S^\nu_r< 0\}$ is the inverse of $S^\nu$ and $\tau_\Omega(x)$ is the life time of $B^x$, $x\in\Omega$. Note that there is a possible intuition for the initial condition in the \emph{past}, as the time parameters of $\phi$ are weighted according to $(S^\nu_{\tau_0(t)}-t)$, which is the waiting/trapping time of the non-Markovian process $t\mapsto B_{S^\nu_{\tau_0(t)}}$.\\
Let us first clarify formula \eqref{eq:SR} for $f=0$. Observe that in \eqref{preRLHO_nl} our operator \eqref{eq:Hnu} is subject to the \emph{exterior/absorbing}  boundary condition $u=\phi$ on $(-\infty,0]\times \Omega$. Also, \eqref{eq:Hnu} is  the generator of the coupled Markov process 
\begin{equation}
r\mapsto \left(t-S^\nu_{r}, B^x_{S^\nu_{r}}\right),\quad (t,x)\in (0,T]\times\Omega.
\label{eq:abs}
\end{equation}
Then we expect the solution to be the absorption of process \eqref{eq:abs} on $(-\infty,0]\times \Omega$, on its first attempt to exit $(0,T]\times \Omega$, which indeed happens at time $\tau_0(t)$ (assuming for simplicity $\tau_\Omega=\infty$). This results in formula  \eqref{eq:SR}. This absorption interpretation can be seen in the more standard case of the fractional Laplacian  with exterior boundary condition \cite[Theorem 1.3]{k17}, or in a general setting in \cite{miao}.\\


Select now time independent initial data $\phi(t)=\phi_0$, $f=0$, $d=1$ and let $r\mapsto B_r$ be a L\'evy process with density $p_r(\cdot)$. Notating $\Phi(dy,dr)=p_r(y)\nu(r)\,dy\,dr$, we can now write 
\begin{equation*}
\text{\eqref{eq:Hnu}}= \int_{\R^+\times\R}\left(\mathbf 1_{\{t-r>0\}}u(t-r,x-y)-u(t,x) \right)\Phi(dy,dr)+\int_{\R} \phi_0(x-y)\Phi(dy,(t,\infty)),
	\label{eq:Hnu_ct}
\end{equation*}
 and the CEE \eqref{preRLHO_nl} is a particular case of \cite[Theorem 4.1, equation (4.1)]{JKMS12}. In \cite{JKMS12}, problem \eqref{preRLHO_nl}  appears in Fourier-Laplace space as
\[
p(\gamma,\xi)=
\frac{1}{\gamma}\frac{\psi_\nu(\gamma+\psi_B(\xi))-\psi_\nu(\psi_B(\xi))}{\psi_\nu(\gamma+\psi_B(\xi))}, \quad \gamma>0,\,\xi\in\R,
\]
and it is shown that the Fourier-Laplace transform of the law of $M_t=B_{S^\nu_{\tau_0(t)}}$ satisfies the above identity, where $\psi_B$ is the Fourier symbol of $B$ and $\psi_\nu$ the Laplace symbol of $S^\nu$. The authors in \cite{JKMS12} also show that $M_t$ arises as the scaling limit of overshoot continuous time random walks (OCTRWs). The overshoot is reflected in the time change living above $t$, in the sense that $S^\nu_{\tau_0(t)}>t$ \cite[III, Theorem 4]{bertoin}. Notice that $M_t$ is trapped precisely when $\tau_0$ is constant, like the fractional-kinetic process \cite[Chapter 2.4]{Meerschaert2012}. But the duration of a waiting  time induced by $\tau_0$ equals the length of the \emph{last} discontinuity of $S^\nu_{\tau_0}$, mirrored in the coupling of space ($B_{S^\nu}$) and time ($\tau_0$). Also, if the subordination is performed by a $\alpha$-stable process $S^\alpha$, then  $M_t$ scales like $B_t$, because $S^\alpha_{\tau_0(t)}=tS^\alpha_{\tau_0(1)}$. The related literature known to us deals with variations of the CEE in  Fourier-Laplace space, mostly motivated by central limit theorems for coupled OCTRWs.  See \cite{MT15,HS11} 
 for multidimensional extensions of OCTRW limits, 
  \cite{MZ15} 
 for explicit densities in certain fractional cases, and \cite{WSJal12,JMS11} 
for alternatives to the first derivative in time.  Due to their peculiar properties OCTRWs are popular models appearing for instance in  physics \cite{Metal10,ZDK15,Z06,FSBZ15,WJM05}, and finance \cite{JWZ09}. 
 Worth mentioning that the OCTRW limit first appeared in \cite{JKMS12}  as the overshooting counterpart of CTRW limits studied in \cite{BMS04,BMM05,MS07}, which result in different CEEs. In this latter case, the counterpart of \eqref{preRLHO_nl} expects the solution to be the subordination of $B$ by $S^\nu_{\tau_0(t)-}$, for $ S^\nu_{s-}$ the left continuous modification of $ S^\nu_{s}$. We could not treat this case, as our method relies on Dynkin formula, and we could not recover a version for the left continuous process $ S^\nu_{s-}$. Note that, although related, problem \eqref{preRLHO_nl} is different from \cite[problem (1.1)]{StingaTorrea17}, as the latter does not impose initial conditions, and in turn it does not describe an anomalous diffusion.  \\
	
To the best of our knowledge, the novel contribution of this article is the following. A general probabilistically natural method to treat wellposedness and stochastic representation for the CEE \eqref{preRLHO_nl} when it features:  initial conditions in the past,  rather general spatial operators on Euclidean domains, a forcing term. Moreover, our proof method tightly follows \cite{DTZ18} and \cite{T18}, which treat the rather different uncoupled EEs of Caputo/Marchaud-type. Therefore proposing a unified method for a large class of fractional/nonlocal EEs with initial conditions in the past,
 without relying on Fourier-Laplace transform techniques.\\
In Theorem \ref{thm:generalised_solution} we prove wellposedness and stochastic representation for \emph{generalised solutions}, which are defined as pointwise limits of solutions to abstract CEEs obtained through semigroup theory. We only assume existence of densities for the Feller process $r\mapsto B_r$, bounded forcing term, but we assume time independent initial conditions  in the domain of the generator of $r\mapsto B_{S^\nu_r}$. In Theorem \ref{thm:weak} we prove that  \eqref{eq:SR} is a \emph{weak solution} for \eqref{preRLHO_nl} for bounded data and $e^{\mathcal L}$ self-adjoint on a bounded domain. We could not prove uniqueness, which appears to be a subtle problem already for the (uncoupled) Marchaud-Caputo EE \cite{Allen17}.\\



The article is organised as follows; Section \ref{sec:pre} introduces general notation, our assumptions, and the main semigroup results  used to treat the operator $H^\nu$; Section \ref{sec:generalised} proves Theorem \ref{thm:generalised_solution} and presents some concrete fundamental solutions to \eqref{preRLHO_nl}; Section \ref{sec:weak} proves Theorem \ref{thm:weak}.  

\section{Notation and subordinated heat operators}\label{sec:pre} We denote by $\R^d,\,\mathbb N,$ $\text{a.e.}$, $a\vee b$ and $a\wedge b$, the   $d$-dimensional Euclidean space, the positive integers, the statement almost everywhere with respect to Lebesgue measure,  the maximum and the minimum between $a,b\in \R$, respectively. We denote by $\Gamma(\beta)$ the Gamma function for $\beta\in (-1,0)\cup (0,\infty)$, and we recall the standard identity $\Gamma(\beta+1)=\Gamma(\beta)\beta$. We write $C_\infty(E)$ for continuous real-valued functions on $E\cup\{\partial\}$, vanishing at infinity on $E$, such that  $f(\partial)=0$, where   $E\cup\{\partial\}$ is the one-point compactification of $E\subset \R^d$. We denote by  $B(E)$ the set of real-valued  bounded measurable functions on $E$. We define the time-space continuous functions spaces for a set $\Omega\subset\mathbb R^d$
\begin{align*}
C_{\partial \Omega}([0,T]\times \Omega):=&\,C_{\infty}([0,T]\times \Omega),\\
C_{0,\partial \Omega}([0,T]\times \Omega):=&\,C_{\partial \Omega}([0,T]\times \Omega)\cap\{f(0)=0\}.
\end{align*}
All the above functions spaces are considered as Banach spaces with the supremum norm. We define $C^1_\infty(-\infty, T]=\{f,f'\in C_\infty(-\infty,T]\}$, with $T\ge0$, and $C^1_c(0,T)=\{f,f'\in C_\infty(\R),\, f\text{ has compact support in }(0,T)\}$, with $T>0$. For two sets of real-valued functions $F$ and $G$ we define 
\[
F\cdot G:=\{f\cdot g:f\in F,\,  g\in G\}.
\]
For a sequence of functions $\{f_n\}_{n\ge1}$ and a function $f$, we write $f_n\to f$ \emph{bpw (bpw a.e.)} if  $f_n$ converges to $f$ pointwise (a.e.) as $n\to\infty$, and the supremum (essential supremum) norms $\|f_n\|_\infty$ are uniformly bounded in $n$. We denote by $L^1(\Omega)$, $L^2(\Omega)$ and $L^\infty(\Omega)$ the standard Banach spaces of integrable, square-integrable and essentially bounded real valued functions on $\Omega$, respectively. In general, we denote by $\|\cdot\|_X$ the  norm of a Banach space $X$, meanwhile the notation  $\|\cdot\|$ is reserved for  the operator norm of a bounded linear operator between Banach spaces. For a set $E\subset \R^d$ we denote by $\overline E$ the closure of $E$ in $\R^d$.

 The notation we use for an $E$-valued stochastic process started at $x\in E $ is $X^x=\{X^x_s\}_{s\ge0}=s\mapsto X^x_s$. Note that the symbol $t$ will often be used to denote the starting point of a stochastic process with state space $E \subset \mathbb R$. By a \emph{ strongly continuous contraction semigroup} $e^{\mathcal G}$ we mean a collection of bounded linear operators $e^{s\mathcal G}:X\to X$, $s\ge0$, where $X$ is a Banach space, such that $e^{(s+r)\mathcal G}=e^{s\mathcal G}e^{r\mathcal G}$, for every $s,r> 0$,  $e^{0\mathcal G}$ is the identity operator, $\lim_{s\downarrow 0}e^{s\mathcal G}f=f$ in $X$, for every $f\in X$,  and $\sup_s\|e^{s\mathcal G}\|\le1$. The generator of $e^{\mathcal G}$ is defined as the pair $(\mathcal G,\text{Dom}(\mathcal G))$, where $\text{Dom}(\mathcal G):=\{f\in X: \mathcal G f:=\lim_{s\downarrow0}s^{-1}(e^{s\mathcal G}f-f) \text{ exists in }X\}$. We say that a set $C\subset \text{Dom}(\mathcal G)$ is a \emph{core for} $(\mathcal G,\text{Dom}(\mathcal G))$ if the generator equals the closure of the restriction of $\mathcal G$ to $C$. Recall that $\text{Dom}(\mathcal G)$ is dense in $X$.
 For a given $\lambda\ge0$ we define the \emph{resolvent of} $e^{\mathcal G}$  by $(\lambda-\mathcal G)^{-1}:=\int_0^\infty e^{-\lambda s}e^{s\mathcal G} \,ds$, and recall that for $\lambda >0$, $(\lambda-\mathcal G)^{-1}:X\to \text{Dom}(\mathcal G)$ is a bijection and it solves the abstract resolvent equation
\begin{equation*}
\mathcal G(\lambda-\mathcal G)^{-1}f=\lambda (\lambda-\mathcal G)^{-1}f-f,\quad f\in X,
\end{equation*}
 see for example \cite[Theorem 1.1]{Dyn65}. By a \emph{Feller semigroup} we mean a strongly continuous contraction semigroup $e^{\mathcal G}$ on any of the (compactified) Banach spaces of continuous functions defined above such that $e^{\mathcal G}$ preserves non-negative functions. 
Feller semigroups are in one-to-one correspondence with Feller processes, where a \emph{Feller process }is a time-homogenous sub-Markov process $\{X_s\}_{s\ge0}$ such that $s\mapsto e^{s\mathcal G}f(x):=\mathbf E[f(X_s)|X(0)=x]$, $f\in X$ is a Feller semigroup \cite[Chapter 1.2]{Schilling}. We recall that every Feller process admits a c\'adl\'ag modification which enjoys the strong Markov property \cite[Theorem 1.19 and Theorem 1.20]{Schilling}, and we always work with such modification. We say that a Feller semigroup is \emph{strong Feller} if $e^{r\mathcal G}$ maps bounded measurable functions to continuous functions for each $r>0$. 

\subsection{The spatial operator $\mathcal L$}

\begin{definition}\label{def:Hs}
We define $(\mathcal L,\text{Dom}(\mathcal L))$ to be the generator of a Feller semigroup $\{e^{r\mathcal L}\}_{r\ge0}$ on $C_\infty(\Omega)$, where the set $\Omega\subset \mathbb R^d$ is either bounded open,  the closure of an open set or compact. We denote the associated Feller process by $B^x=s\mapsto B_s^x$, when started at $x\in \Omega$.  As usual, the Feller process $s\mapsto B_s^x$ is defined to be in the cemetery if $s\ge \tau_\Omega(x)$, defining the life times $\tau_\Omega(x)=\inf\{s>0:B^x_s\notin \Omega\}$, $x\in\Omega$, so that $B^x_s=B^x_{s\wedge \tau_\Omega(x)}$.
\end{definition}
We will use the following assumption for the spatial semigroup $e^{\mathcal L}$.
\begin{description}
\item[(H1)\label{H1}] The operator  $e^{r\mathcal L}$ allows a density with respect to Lebesgue measure for each $r>0$, which we denote by $y\mapsto p^\Omega_r(x,y)$, $x\in \Omega$.
\end{description}

\begin{definition}
For a Feller process in $\R^d$, we say that $\Omega$ is a \emph{regular set} if  $\Omega\subset\R^d$  is open, and for each $z\in\partial \Omega$, $\mathbf P[\tau_\Omega(z)=0]=1$. Here $\partial \Omega$ denotes the Euclidean boundary of $\Omega$.
\end{definition}

\begin{example}\label{ex:spati} We mention some examples of Feller processes that satisfy  \ref{H1}, including several nonlocal and fractional derivatives on $\mathbb R^d$ and on bounded domains.
\begin{enumerate}[(i)]
\item  Diffusion processes in $\Omega=\R^d$ with generator $\dive(A(x)\nabla)$, where $A:\R^d \to\R^d \times\R^d $ is a  matrix valued function which is bounded, measurable, positive, symmetric and uniformly elliptic \cite[Theorem II.3.1, p. 341]{Stroock88}. Moreover the density $(t,x,y)\mapsto p^\Omega_t(x,y)$ is continuous on $(0,\infty)\times\R^d \times\R^d $, and the induced Feller semigroup is strong Feller (which follows by the Aronson estimate \cite[formula (I.0.10)]{Stroock88}).

\item All strong Feller L\'evy processes ($\Omega=\mathbb R^d$). Indeed this is a characterisation \cite[Lemma 2.1, p.338]{hawkes1979potential}. See \cite[Chapter 5.5]{Kuhn17} for a discussion. This class includes all stable L\'evy processes.
	
\item Possible conditions on L\'evy-type  or L\'evy measures $\kappa(x,dy)$ ($\Omega=\mathbb R^d$)  are
	\begin{enumerate}[(a)]
\item kernels $\kappa(dy)$ for $d=1$ such that $\kappa(dy)\ge y^{-1-\alpha}dy$ for all small $y$ \cite[Proposition 28.3]{Sato};
\item kernels $\kappa(y)$, such that $\int_{\R^d \backslash\{0\}} \kappa(y)\,dy=\infty$ \cite[Theorem 27.7]{Sato};
\item kernels $\kappa(x,y)$ such that the respective symbols satisfies the H\"older continuity-type conditions  in \cite[Theorem 2.14]{Kuhn17}, and see also \cite[Theorem 3.3]{Kuhn17}.
\end{enumerate}
		\item Clearly any Feller processes $X$ taking values in $\mathbb R^d$ such that its density is continuous. If $X$ is also strong Feller and  $\Omega\subset \mathbb R^d$ is a regular set, then the process killed upon the first exit from $\Omega$ is a Feller process on $\Omega$  \cite[p. 68]{chung1986doubly}, and it has a continuous density (which can be proved by the strong Markov property as in \cite[formula (4.1)]{MeerChen}). This case includes the regional fractional Laplacian $(-\Delta)^\beta_\Omega$ \cite{MeerChen}.
	
		\item Any subordination of a Feller process by a L\'evy subordinator which itself satisfies \ref{H1}, which is a straightforward consequence of \cite[Theorem 4.3.5]{jacob2001pseudo}. This case includes the spectral fractional Laplacian $(-\Delta_\Omega)^\beta$ \cite{BV17,Bogdan}.
	
	\item We mention the articles \cite{chen2010two, grzywny2018heat} and references therein for related discussions about  some jump-type generators with symmetric and non-symmetric kernels.
	
	\item The 1-$d$ reflected Brownian motion \cite[Chapter 6.2]{BW}, so that $\Omega=[0,\infty)$, and $\mathcal L =\partial_x^2$, endowed with the Neumann boundary condition on $(0,T]\times\{0\}$.

	\item The restriction to $C_\infty(\overline\Omega)$ of the $L^2(\Omega)$ semigroup generated by the divergence operator $\dive(A(x)\nabla)$ with Neumann boundary conditions on a Lipschitz open bounded connected set $\Omega\subset\R^d$, for the same coefficients $A$ as in Example \ref{ex:spati}-(i).  This is a consequence of \cite[Theorem 3.10, Section 2.1.2]{GP11}.

	
	\item The reflected spectrally negative $\beta$-stable L\'evy process on  $\Omega=[0,\infty)$, for $\beta\in(1,2)$ \cite[Theorem 2.1, Corollary 2.4]{MB15}. In this case 
	\[
	\mathcal L u(x) = \partial_x^\beta u(x)=\int_0^xu''(y)\frac{(x-y)^{1-\beta}}{\Gamma(2-\beta)}\,dy,\quad x>0,
	\]
	 for $u$ in the core given in \cite[Theorem 2.1]{MB15}, which features $u'(0)=0$ at 0. Note that $\partial_x^\beta$ is the Caputo derivative of order  $\beta\in(1,2)$ \cite{kai}. Interestingly \cite[Theorem 2.3]{MB15}, the corresponding forward equation satisfies a fractional Neumann boundary condition $D_{-y}^{\beta-1}u(0)=0$, where $D_{-y}^{\beta-1}$ is the Marchaud derivative 
	\[
	D_{-y}^{\beta-1}u(y)=\partial_y\int_y^\infty u(r)\frac{(r-y)^{1-\beta}}{\Gamma(2-\beta)}\,dr,\quad y\ge0.
	\] 
	\end{enumerate}
\end{example}

For our notion of weak solution in Section \ref{sec:weak} we will use a stronger assumption for the spatial semigroup. Namely:
\begin{description}
\item[(H1')\label{H1'}] the set $\Omega$ is a bounded open subset of $\R^d$, and   $e^{\mathcal L}$ is a Feller semigroup on $X=C_\infty(\Omega)$ or $X=C_\infty(\overline\Omega)$ such that assumption \ref{H1} holds, and $e^{\mathcal L}$ is self-adjoint, in the sense that for each $r>0$
\begin{equation}
\int_\Omega e^{r\mathcal L}v(x)\,w(x)\,dx=\int_\Omega v(x)\,e^{r\mathcal L}w(x)\,dx,\quad v,w\in X.
\label{eq:sa}
\end{equation}
\end{description}

\begin{example}\label{ex:H10}
\begin{enumerate}[(i)]
	\item Assumption \ref{H1'} holds for several processes obtained by killing a Feller process on $\R^d$ upon exiting  a regular bounded domain $\Omega$.  This is for example the case of the Dirichlet Laplacian $-\Delta_\Omega$, the regional fractional Laplacian $(-\Delta)^{\beta}_\Omega$ and the  spectral fractional Laplacian $(-\Delta_\Omega)^{\beta}$, $\beta\in(0,1)$. These killed semigroups are Feller, as explained in Example \ref{ex:spati}-(iv)-(v). Property \eqref{eq:sa} follows by the eigenfunction decomposition  of the $L^2(\Omega)$ extension of the killed Feller semigroup \cite{Davies,MeerChen,Bogdan}, along with $C_\infty(\Omega)\subset L^2(\Omega)$. More generally, one can use the theory  regular symmetric Dirichlet forms, for example combining \cite[Proposition 3.15]{Schilling} with \cite[Corollary 3.2.4-(ii)]{ChenFuk12}.  This examples correspond to $0$ boundary conditions on $\partial \Omega$ or $\Omega^c$.
	
\item 
Assumption \ref{H1'} holds for the Feller semigroup of Example \ref{ex:spati}-(viii), as an immediate consequence of the semigroup being generated by a (symmetric) regular Dirichlet form \cite[Theorem 3.10]{GP11}. One can also consider an appropriate subordination of the Feller semigroup of Example  \ref{ex:spati}-(viii), as mentioned in Example \ref{ex:spati}-(v). Then \ref{H1} still holds along with property \eqref{eq:sa}, which can be seen by applying the eigenfunction expansion to the subordinated semigroup. This examples correspond Neumann boundary conditions on $\partial \Omega$.
\end{enumerate}
\end{example}

\begin{remark}
We could allow $\Omega=\R^d$ in assumption \ref{H1'},  but it would affect the clarity of the exposition, as  we would have to consider extra cases in several steps in Section \ref{sec:weak}.
\end{remark}



\subsection{Subordinators and subordinated heat operators}
We will always assume the following.
\begin{description}
\item[(H0)\label{H0}] Denote by $\nu:(0,\infty)\to [0,\infty)$ any  continuous function such that 
\[
\int_0^\infty (r\wedge 1)\nu(r)\,dr<\infty\quad\text{and}\quad\int_0^\infty \nu(r)\,dr=\infty.
\]
\end{description}

\begin{definition}\label{def:Hsnu}
 We denote
 by $S^\nu=\{S^\nu_r\}_{r\ge0}$  \emph{the L\'evy subordinator for} $\nu$, characterised by the log-Laplace transforms $\log\mathbf E[e^{-kS^\nu_r}]=r\int_0^\infty (e^{-ks}-1)\nu(s)\,ds$, for $r,k>0$. We define the first exit/passage times 
\[
\tau_0(t):=\inf\{r>0:S^\nu_r>t\},\quad t>0.
\]
\end{definition}

\begin{remark}\label{rmk:dens}
\begin{enumerate}[(i)]
\item Recall that for each  $r>0$, the random variable $S^\nu_r$ allows a density \cite[Theorem 27.7]{Sato}, which we denote by  $p^\nu_r$.

	\item 
Recall that for every $t\in (0,T]$
$$
\int_0^\infty\int_0^tp^\nu_s(t-z)\,dz\,ds=\mathbf E[\tau_0(t)]\le \mathbf E[\tau_0(T)]<\infty,
$$ 
see for example \cite[Theorem 19 and page 74]{bertoin}. In particular $\sup_{t\in(0,T]}\mathbf E[\tau_0(t)]<\infty$.

\item  To obtain the stable subordinator case select
$$
\nu(r):= r^{-1-\alpha}/|\Gamma(-\alpha)|,\quad r>0,\quad\alpha\in(0,1),
$$
then $S^\nu=S^\alpha$ is the  the $\alpha$-stable subordinator, characterised by the Laplace transforms $\mathbf E[e^{-kS^\nu_r}]=e^{-rk^\alpha}$, for $r,k>0$. Denoting its densities  by $p^\alpha_s$, $s>0$, recall that 
$$
\mathbf E[\tau_0(t)]=\frac{t^\alpha}{\Gamma(\alpha+1)},
$$ 
see for example \cite[Example 5.8]{Bogdan}. 
\item We refer to \cite[Chapter 5.2.2]{Bogdan} for examples of  subordination kernels $\nu$.
\end{enumerate}
\end{remark}

We define three semigroups that correspond to three different space-time valued of processes related to the heat operator $-\partial_t+\mathcal L$. Namely the `free' process $s\mapsto(t-s,B^x_s)$, the `absorbed at 0' process $s\mapsto((t-s)\vee 0,B^x_s)$, and the `killed at 0' process $s\mapsto((t-s),B^x_s)$ for $t>s$ and $\partial$ otherwise. It is straightforward to prove that such semigroups are Feller and we omit the proof. 
\begin{definition}\label{def:HSG} Define the operators  $e^{s(-\partial_t)}u(t):=u(t-s)$ and $e^{s(-\partial_{t,0})}u(t):=u((t-s)\vee 0)$, $t\in \mathbb R$, $s\ge0$, acting on the time variable. With  the semigroup $e^{\mathcal L}$ acting   on the $\Omega$-variable, define the three Feller semigroups
\begin{align*}
e^{s\mathcal H}&:= e^{s(-\partial_t)}e^{s\mathcal L},\quad &\text{on }&C_\infty((-\infty,T]\times\Omega),\quad &s&\ge0,\\
e^{s\mathcal H_0}&:= e^{s(-\partial_{t,0})}e^{s\mathcal L},\quad &\text{on }&C_{\partial \Omega}([0,T]\times\Omega),\quad &s&\ge0,\\
e^{s\mathcal H_0,\text{kill}}&:= e^{s\mathcal H_0},\quad &\text{on }&C_{0,\partial \Omega}([0,T]\times\Omega),\quad &s&\ge0,
\end{align*}
with the respective generators denoted by 
\[
 (\mathcal H,\text{Dom} (\mathcal H ) ),\quad (\mathcal H_0,\text{Dom} (\mathcal H_0 ) ),\quad\text{and}\quad (\mathcal H_0^{\text{kill}},\text{Dom} (\mathcal H_0^{\text{kill}} ) ).
\]  
\end{definition}
\begin{remark}\label{rmk:on_HSG}
 Note that 	
	\begin{align*}
	e^{r\mathcal H}u(t,x)=e^{r\mathcal L}u(t-r,x)=\mathbf E\left [u(t-r, B^x_r)\right],\quad \text{and}\quad e^{0\mathcal H}u(t,x)=u(t,x).
	\end{align*}
	\end{remark}

We now define three semigroups that respectively  correspond to subordinating  the three semigroups in Definition \ref{def:HSG} by an the independent L\'evy subordinator $S^\nu$.

\begin{definition}\label{def:Hs3}
For appropriate functions $u$, we define for $r>0$ 
\begin{align}\label{eq:SGH}
e^{r\mathcal H^\nu}u(t,x)&=\int_0^\infty e^{s\mathcal H}u(t,x)\,p^\nu_r(s)\,ds, & t\in\R,\\ \label{eq:SGHO}
e^{r\mathcal H^\nu_0}u(t,x)&=\int_0^t e^{s\mathcal H}u(t,x)\,p^\nu_r(s)\,ds+ \int_t^\infty e^{s\mathcal L}u(0,x) p^\nu_r(s)\,ds, &t\in[0,T],\\ \label{eq:SGHOkill}
e^{r\mathcal H^\nu_0,\text{kill}}u(t,x)&=\int_0^t e^{s\mathcal H}u(t,x)\,p^\nu_r(s)\,ds,& t\in(0,T],
\end{align}
and  $e^{r\mathcal H^\nu}u(t,x)=e^{r\mathcal H^\nu_0}u(t,x)=e^{r\mathcal H^\nu_0,\text{kill}}u(t,x)= u(t,x),$ for $r=0$.
\end{definition}
\begin{remark}\label{rmk:on_HOSG}
\begin{enumerate}[(i)]
	\item If $u(0)=0$, then $e^{r\mathcal H^\nu_0}u(t,x)=e^{r\mathcal H^\nu_0,\text{kill}}u(t,x)$, and note that for each $r>0$, $B((0,T]\times\Omega)$ is invariant under $e^{r\mathcal H^\nu_0,\text{kill}}$.
	\item If $u$ is independent of time, then 
	$$e^{r\mathcal H^\nu_0}(u)(t,x)=\int_0^\infty e^{s\mathcal L}u(x)\,p^\nu_r(s)\,ds=\mathbf E\left[u\left(B^x_{S^\nu_r}\right)\right]
	$$ is independent of time.
	\end{enumerate}
\end{remark}

The next theorem shows that the operators in Definition \ref{def:Hs3} define Feller semigroups, it gives a pointwise representation for the generators on `nice' cores, and finally it connects the domains of the generators of $e^{r\mathcal H^\nu_0}$ and $e^{r\mathcal H^\nu_0,\text{kill}}$. These statements serve various purposes, but let us outline our main line of thinking. Our strategy is to reduce \eqref{preRLHO_nl} to \eqref{postRLHO} with an appropriate forcing term, as suggested by the simple Lemma \ref{lem:cap_to_mar} (here we use the generators pointwise representation). Hence we solve problem \eqref{postRLHO} in the framework of abstract resolvent equations (Theorem \ref{thm:generalised_solution}). To do so, we use Theorem \ref{thm:HO_SG}-(iv) to reduce problem \eqref{postRLHO} to the 0 initial condition version, easily solved by inverting $\mathcal H^{\nu,\text{kill}}_0$ (Lemma \ref{lem:sdg}). Moreover, Theorem \ref{thm:HO_SG} allows us to access Dynkin formula.   
\begin{theorem}\label{thm:HO_SG} Assume \ref{H0} and let $T\in (0,\infty)$. With the notation of Definition \ref{def:Hs} and Definition \ref{def:Hs3}:
\begin{enumerate}[(i)]
	\item The operators  $e^{r\mathcal H^\nu}$, $r\ge 0$ form a Feller semigroup on $C_{\infty}((-\infty,T]\times\Omega)$. We denote the generator of the semigorup by $\left(\LHM,\text{Dom}(\LHM)\right). $\\
Moreover, $\text{Dom}(\mathcal H)$ is a core for $\left(\LHM,\text{Dom}(\LHM)\right)$, and for  $g\in \text{Dom}(\mathcal H)$
\begin{equation}
\LHM g(t,x)=\LHMp g(t,x):=\int_0^\infty\left(e^{r\mathcal H}g(t,x)-g(t,x)\right) \nu(r)\,dr.
\label{eq:LHM_c}
\end{equation}

	\item The operators $e^{r\mathcal H^\nu_0}$, $r\ge 0$ form a Feller semigroup on  $C_{\partial\Omega}([0,T]\times\Omega)$. We denote the generator of the semigorup  by $\left(\LH ,\text{Dom}(\LH )\right). $\\
Moreover, $\text{Dom}(\mathcal H_0)$ is a core for $(\LH,\text{Dom}(\LH))$, and  
\begin{equation}
\LH g(t,x)=\LHp g(t,x),\quad\text{for } g\in \text{Dom}(\mathcal H_0),
\label{eq:LH_c}
\end{equation}
where
\begin{equation*}
\LHp g(t,x):=\int_0^t\left(e^{r\mathcal H}g(t,x)-g(t,x)\right) \nu(r)\,dr+\int_t^\infty\left(e^{r\mathcal L}g(0,x)-g(t,x)\right) \nu(r)\,dr.
\end{equation*}

\item The operators  $e^{r\mathcal H^\nu_0,\text{kill}}$, $r\ge 0$ form a Feller semigroup on  $C_{0,\partial\Omega}([0,T]\times\Omega)$. We denote the generator of the semigorup  by $(\LHkill ,\text{Dom}(\LHkill )). $\\
Moreover,  $\text{Dom}(\mathcal H_0^{\text{kill}})$ is a core for $(\LHkill,\text{Dom}(\LHkill))$, and 
\[
\LHkill g= \LHp g,\quad\text{for } g\in \text{Dom}(\mathcal H_0^{\text{kill}}).
\]

\item In addition, it holds that $\LH  = \LHkill $  on $\text{Dom}(\LHkill )$, and
\begin{equation}\label{eq:dom_relation}
\text{Dom}(\LHkill ) = \text{Dom}(\LH )\cap \{f(0)=0\}.
\end{equation}

\end{enumerate}
\end{theorem}
\begin{proof} 
The statements  (i), (ii) and (iii) are all consequences of \cite[Theorem 4.3.5 and Proposition 4.3.7]{jacob2001pseudo} along  with  preservation of positive functions and the contraction property, which are easily checked directly from the definitions \eqref{eq:SGH}, \eqref{eq:SGHO} and \eqref{eq:SGHOkill}, respectively.\\

iv) To prove \eqref{eq:dom_relation}, we note that the inclusion `$\subset$' is clear because,  $\text{Dom}(\LHkill )\subset C_{0,\partial\Omega}([0,T]\times\Omega)$, and the two semigroups \eqref{eq:SGHO} and \eqref{eq:SGHOkill} agree on $C_{0,\partial\Omega}([0,T]\times\Omega)$ by Remark \ref{rmk:on_HOSG}-(i). For the opposite inclusion `$\supset$', we show that 
$$
\text{if }g\in \text{Dom}(\LH ),\text{ then }g-g(0)\subset \text{Dom}(\LHkill ).
$$
Consider the resolvent representation for $g$ for a given $\lambda>0$ and $g_\lambda\in C_{\partial\Omega}([0,T]\times\Omega)$ given by
\[
g(t,x)=\int_0^\infty e^{-r\lambda} e^{r\mathcal H^\nu_0}g_\lambda (t,x)\,dr,
\]
 and 
\[
g(0,x)=\int_0^\infty e^{-r\lambda} e^{r\mathcal H^\nu_0}g_\lambda (0,x)\,dr=\int_0^\infty e^{-r\lambda} e^{r\mathcal H^\nu_0}(g_\lambda(0) )(t,x)\,dr,
\]
where we use Remark \ref{rmk:on_HOSG}-(ii). Then 
\begin{align*}
g(t,x)-g(0,x)= \int_0^\infty e^{-r\lambda} e^{r\mathcal H^\nu_0}(g_\lambda -g_\lambda(0))(t,x)\,dr\in \text{Dom}(\LHkill )
\end{align*} 
as $g_\lambda -g_\lambda(0)\in C_{0,\partial\Omega}([0,T]\times\Omega)$ and $e^{r\mathcal H^\nu_0}=e^{r\mathcal H^\nu_0,\text{kill}}$  on $C_{0,\partial\Omega}([0,T]\times\Omega)$.\\
We can now conclude equating resolvent equations, as for any $g\in \text{Dom}(\LHkill )$, for a positive $\lambda>0$ and a respective $g_\lambda\in  C_{0,\partial\Omega}([0,T]\times\Omega)$ 
\[
\LHkill  g = \lambda g-g_\lambda = \LH  g. 
\]
\end{proof}

\begin{remark}
Let us stress that Theorem \ref{thm:HO_SG}-(iv), although unsurprising,  is a vital technical ingredient for this work. This is because it allows to obtain uniqueness of our notion of a \emph{solution in the domain of the generator } for \eqref{postRLHO} (see the proof of Lemma \ref{lem:sdg}-(i)). Such notion of solution is our building block for weak solutions to \eqref{preRLHO_nl} in Section \ref{sec:weak}.
\end{remark}


\begin{example}
Concerning Theorem \ref{thm:HO_SG}, if $\mathcal L =\Delta,$ and $\Omega=\R^d $, then, using standard notation,
\begin{align*}
\text{Dom}(\mathcal H)&=C^{1,2}_\infty((-\infty,T]\times\R^d),\\
\text{Dom}(\mathcal H_0)&=C^{1,2}_{\infty}([0,T]\times\R^d), \\
\text{Dom}(\mathcal H_0^{\text{kill}})&=C^{1,2}_{\infty}([0,T]\times\R^d)\cap\{f(0)=0\}.
\end{align*} 
\end{example}
\begin{remark}
To see that $H^\nu u$ is well defined pointwise for  $u\in \text{Dom}(\mathcal H)$ one can use the general bound in Remark \ref{rmk:smallr} along with \ref{H0}.
\end{remark}

The proof of Theorem \ref{thm:HO_SG}-(i) guarantees that the next definition make sense.
\begin{definition}\label{def:L_Omega_alpha}
We denote by $(\mathcal L^{\nu},\text{Dom}(\mathcal L^{\nu}))$  the generator of the Feller semigroup 
$$
e^{r\mathcal L^{\nu}}(\cdot) :=\int_0^\infty e^{s\mathcal L}(\cdot) p^\nu_r(s)\,ds,\quad r>0,
$$ on $C_\infty(\Omega)$ induced by the Feller process $r\mapsto B_{S^\nu_r}$. 
\end{definition}
\begin{remark}\label{rmk:lifetime}
The  life time of the Feller process $r\mapsto B_{S^\nu_r}$ is 
\[
\inf\{s>0: B^x_{S^\nu_s}\notin\Omega\}=\inf\{s>0: S^\nu_s \ge \tau_\Omega(x)\}=\tau_0(\tau_\Omega(x)),
\]
for each $x\in\Omega$, where the first equality follows by $B^x_{s}=B^x_{s\wedge \tau_\Omega(x)}$, and its independence with respect to $S^\nu_s$.
\end{remark}
We will later use the following simple lemma.

\begin{lemma}\label{lem:phi_0}
Suppose  $\phi_0\in \text{Dom}(\mathcal L^{\nu})$ and constantly extend $\phi_0(x)$ to $[0,T]$ for each $x\in\Omega$.  Then $\phi_0\in \text{Dom}(\LH)\subset C_{\partial\Omega}([0,T]\times\Omega)$ and
\[
\LH  \phi_0 =\mathcal L^{\nu}\phi_0.
\]
\end{lemma}
\begin{proof}
This is straightforward, because 
\begin{align*}
r^{-1}\left(e^{r\mathcal H^\nu_0}(\phi_0)(t,x)-\phi_0(t,x)\right)& =r^{-1}\left( \int_0^\infty e^{s\mathcal L}\phi_0(x) p^\nu_r(s)\,ds-\phi_0(x)\right)\\
& =r^{-1}\left( e^{r\mathcal L^{\nu}}\phi_0(x) -\phi_0(x)\right)\to \mathcal L^{\nu}\phi_0,
\end{align*}
as $r\downarrow 0$, uniformly in both $t,$ and $x$. 
\end{proof}

\section{Generalised solution for time-independent initial condition }\label{sec:generalised}

We prove existence, uniqueness and stochastic representation for generalised solutions to the `Caputo-type' problem
	\begin{equation}\label{postRLHO}
\left\{
\begin{split}
\mathcal H^\nu_0 u(t,x)    &=-g, & \text{in }& (0,T]\times \Omega,\\ 
  u(0,x)&=\phi_0(x),  &\text{in }&\{0\}\times \Omega,
\end{split}
\right.
\end{equation}
under assumptions \ref{H0} and \ref{H1}.
In particular, we will obtain the Feynman-Kac formula 
\begin{equation}
u(t,x)=\mathbf E\left[\phi_0\left(B^x_{S^\nu_{\tau_0(t)}}\right)\mathbf1_{\{\tau_0(t)<\tau_0(\tau_\Omega(x))\}} \right] +\mathbf E\left[\int_0^{\tau_0(t)\wedge\tau_0(\tau_\Omega(x))}g\left(t-S^\nu_r,B^x_{S^\nu_r}\right) \,dr\right],
\label{eq:SRHO}
\end{equation}
for the solution to \eqref{postRLHO}.
\begin{remark}\label{rem:SRbry} Recalling Remark \ref{rmk:lifetime}, observe that if $g(\partial)= 0$ for $\partial$  the  cemetery state of $C_{\partial\Omega}([0,T]\times\Omega)$, then
\begin{align*}
\mathbf E\left[\int_0^{\tau_0(t)\wedge\tau_0(\tau_\Omega(x))}g\left(t-S^\nu_r,B^x_{S^\nu_r}\right) \,dr\right]&=\mathbf E\left[\int_0^{\tau_0(t)}g\left(t-S^\nu_r,B^x_{S^\nu_r \wedge \tau_\Omega(x)}\right) dr\right]\\
&=\int_0^\infty\mathbf E\left[ \mathbf 1_{\{t-S^\nu_r>0\}}g\left(t-S^\nu_r,B^x_{S^\nu_r}\right) \right]dr.
\end{align*}
Similarly, if $\phi_0(\partial)=0$, for $\partial$ the  cemetery state of $C_\infty(\Omega)$, then
\begin{align*}
\mathbf E\left[\phi_0\left(B^x_{S^\nu_{\tau_0(t)}}\right) \right]&= \mathbf E\left[\phi_0\left(B^x_{S^\nu_{\tau_0(t)}\wedge\tau_\Omega(x) }\right) \right]= \mathbf E\left[\phi_0\left(B^x_{S^\nu_{\tau_0(t)}}\right)\mathbf1_{\{\tau_0(t)<\tau_0(\tau_\Omega(x))\}}  \right].
\end{align*}
\end{remark}

\begin{remark}
Problem \eqref{postRLHO} formally corresponds to problem \eqref{preRLHO_nl} for time independent initial condition $\phi(t)=\phi_0$, in a similar way as Caputo and Marchaud evolution equations are related in \cite{T18}.
\end{remark}

We first assume some compatibility condition on the forcing term and the initial data in order to construct the following kind of strong solution.

\begin{definition}
The function $u$ is a  \emph{ solution in the domain of generator to } \eqref{postRLHO} if 
\begin{equation}
\LH  u=-g,\,\text{on }(0,T]\times\Omega,\quad u(0)=\phi_0,\quad u\in \text{Dom}(\LH ).
\label{eq:def_sdg}
\end{equation}
\end{definition}

\begin{lemma}\label{lem:sdg}
Assume \ref{H0}, and let  $g\in  C_{\partial\Omega}([0,T]\times\Omega)$ and
 $\phi_0\in \text{Dom}(\mathcal L^{\nu})$ such that $g+\mathcal L^{\nu}\phi_0\in  C_{0,\partial\Omega}([0,T]\times\Omega)$. 
\begin{enumerate}[(i)]
	\item Then there exists a unique solution in the domain of the generator to \eqref{postRLHO}.
\item  Moreover, the solution in the domain of the generator allows the stochastic representation \eqref{eq:SRHO}.
\end{enumerate}
\end{lemma}
\begin{proof} 

i) We first claim that
\[
(-\LHkill )^{-1} (g+\mathcal L^{\nu}\phi_0)= \int_0^\infty \left(\int_0^t e^{s\mathcal H}(g+\mathcal L^{\nu}\phi_0)\,p^\nu_r(s)\,ds\right)dr,
\]
is the  unique solution to the abstract evolution equation 
\begin{equation}
\LHkill  u=-g-\mathcal L^{\nu}\phi_0,\,\text{on }(0,T]\times\Omega,\quad u(0)=0,\quad u\in \text{Dom}(\LHkill).
\label{eq:def_sdg_0}
\end{equation}
 Let $f\in  C_{0,\partial\Omega}([0,T]\times\Omega)$. Then
\begin{align*}
(-\LHkill )^{-1}f(t,x)&=\int_0^\infty e^{r\mathcal H^\nu_0} f(t,x)   \,dr\\
&= \int_0^\infty \left(\int_0^t e^{s\mathcal H}f(t,x)\,p^\nu_r(s)\,ds\right)  \,dr\\
&\le \|f\|_\infty\int_0^\infty\int_0^tp^\nu_r(s)\,ds\,dr\\
&\le \|f\|_\infty\mathbf E[\tau_0(T)].
\end{align*}
Moreover, using $e^{r\mathcal H^\nu_0} f\in C_{0,\partial\Omega}([0,T]\times\Omega)$,  Dominated Convergence Theorem (DCT) proves that $(-\LHkill )^{-1}$ maps $C_{0,\partial\Omega}([0,T]\times\Omega)$ into itself. Then \cite[Theorem 1.1']{Dyn65}
 proves the claim.\\
Recall that  by Theorem \ref{thm:HO_SG}-(iv) 
$$
\LH  \tilde u = \LHkill  \tilde u, \quad\text{if }\tilde u\in \text{Dom}(\LHkill)=\text{Dom}(\LH )\cap \{f(0)=0\}. 
$$ 
 It is now enough to show that $\tilde u=u-\phi_0$ is a solution to \eqref{eq:def_sdg_0} if and only if $u$ is a solution  to \eqref{eq:def_sdg}. For the `only if' direction, define 
\begin{equation*}
u:= \tilde u+\phi_0.
\label{eq:}
\end{equation*}
Then  $u\in \text{Dom}(\LH )$ as $\tilde u\in \text{Dom}(\LH )$  by Theorem \ref{thm:HO_SG}-(iv) and $\phi_0\in \text{Dom}(\LH )$ by Lemma \ref{lem:phi_0}, and $u$ solves 
$$
\LH  (\tilde u+\phi_0)= \LH  \tilde u+\mathcal L^{\nu} \phi_0=-g,
$$
along with $u(0)=\phi_0$. The `if' direction is similar and omitted. \\

ii) Fix $(t,x)\in (0,T]\times\Omega$. First compute 
 \begin{align*}
\int_0^\infty e^{r\mathcal H^\nu_0,\text{kill}}(\mathcal L^{\nu}\phi_0)(t,x)\,dr=&\, \mathbf E\left[\int_0^\infty \int_0^\infty \mathcal L^{\nu}\phi_0(B^x_s)\mathbf 1_{\{t-s>0\}}p_r^\nu(s)\,ds\,dr\right ]\\
=&\, \mathbf E\left[\int_0^\infty  \mathcal L^{\nu}\phi_0(B^x_{S^\nu_r})\mathbf 1_{\{t-S_r^\nu>0\}}\,dr\right ]\\
=&\,  \mathbf E\left[\int_0^{\tau_0(t)} \mathcal L^{\nu}\phi_0(B^x_{S^\nu_r})\,dr\right],
\end{align*}
where we use $\{t-S_r^\nu>0\}=\{\tau_0(t)>r\}$, by the monotonicity of the subordinator $S^\nu$. By the integrability of  $\tau_0(t)$, we can  apply  Dynkin formula \cite[Corollary of Theorem 5.1]{Dyn65}  to obtain
$$
 \mathbf E\left[\int_0^{\tau_0(t)} \mathcal L^{\nu}\phi_0\left(B^x_{S^\nu_r}\right)\,dr\right]+\phi_0(x) =  \mathbf E\left[\phi_0\left(B^x_{S^\nu_{\tau_0(t)}}\right)\right].
$$
This proves that $u$ can be written as \eqref{eq:SRHO}.\\
\end{proof}

We now give another definition of solution as the pointwise limit of solutions in the domain of the generator. This allows us to drop the compatibility  condition on the data in Lemma \ref{lem:sdg}. We  pay  a price by assuming \ref{H1}. 

\begin{definition}\label{def:gen_sol}
Let $g\in L^\infty((0,T)\times\Omega)$ and let $\phi_0\in \text{Dom}(\mathcal L^{\nu})$. Then $u$ is a \emph{generalised solution to }\eqref{postRLHO} if 
\[
u=\lim_{n\to\infty}u_n,\quad\text{pointwise},
\]
where $\{u_n\}_{n\ge1}$ is the sequence of solutions in the domain of the generator to  \eqref{postRLHO} for respective forcing terms $\{g_n\}_{n\ge1} \in C_{\partial\Omega}([0,T]\times\Omega)$ such that $g_n (0)=\mathcal L^{\nu}\phi_0 $ for all $n\ge1$, $g_n\to g$ bpw a.e..
\end{definition}

\begin{theorem}\label{thm:generalised_solution}
Assume \ref{H0}, \ref{H1}  and let $g\in L^\infty((0,T)\times\Omega)$, $\phi_0\in \text{Dom}(\mathcal L^{\nu})$. Then there exist a unique generalised solution to \eqref{postRLHO}. Moreover the generalised solution allows the stochastic representation \eqref{eq:SRHO}.
\end{theorem}
\begin{proof}
Take a sequence $\{g_n\}_{n\ge1}$ as in Definition \ref{def:gen_sol}. Then the respective solution in the domain of the generator $u_n$ allows the representation \eqref{eq:SRHO}, for $g\equiv g_n$. Fix $(t,x)\in(0,T]\times\Omega$. By assumption \ref{H1}, Remark \ref{rmk:dens}-(i), Remark \ref{rem:SRbry}, and independence of $S^\nu_r$ and $B^x_r$, we can rewrite the second term in \eqref{eq:SRHO} as
\begin{align*}
F(g_n):&=\int_0^\infty \left( \int_\Omega \int_0^\infty \mathbf 1_{\{t-s>0\}}g_n(t-s,y)p^\Omega_s(x,y)\, p^\nu_r(s)\,ds\,dy  \right)dr\\
&= \int_\Omega \int_0^\infty g_n(t-s,y)\left( \mathbf 1_{\{t-s>0\}} p^\Omega_s(x,y)\,\int_0^\infty p^\nu_r(s)\,dr\right)ds\,dy.
\end{align*}
By DCT, $F(g_n)\to F(g)$, as $n\to\infty$, using the  dominating function 
\[
(s,y)\mapsto\sup_n\|g_n\|_\infty \mathbf1_{\{t>s\}}p^\Omega_s(x,y) \int_0^\infty p^\nu_r(s)\,dr,
\] 
given that $F(|g_n|)\le \sup_n\|g_n\|_\infty \mathbf E[\tau_0(t)]$. Hence a generalised solution exists and it permits the stochastic representation \eqref{eq:SRHO}. Conclude observing that independence of the approximating sequence proves uniqueness.

\end{proof}

\begin{remark}\label{rem:convergenceSR}
By definition, a sequence $u_n$ of solutions in the domain of the generator converges pointwise to the generalised solution $u$ on $(0,T]\times\Omega$.  Moreover, by the stochastic representation \eqref{eq:SRHO}, 
\[
\sup_n \|u_n\|_{B([0,T]\times\Omega)} \le \|\phi_0\|_{B(\Omega)} + \sup_n\|g_n\|_{B([0,T]\times\Omega)} \mathbf E[\tau_0(T)] <\infty,
\]
where each $g_n$ is the data of the solution in the domain of the generator $u_n$.
\end{remark}

\begin{remark} 
We refer to Example \ref{ex:spati} for possible choices of domain $\Omega$ and generator $(\mathcal L,\text{Dom}(\mathcal L))$. 
\end{remark}

We now show that the fundamental solution that defines \eqref{eq:SR} allows a density with respect to Lebesgue measure.

\begin{lemma} \label{lem:densitystau}
Assume \ref{H0}. Then for each $t>0$, the random variable $S^\nu_{\tau_0(t)}-t$ allows a density supported on $(0,\infty)$, and we can write the density for almost every  $r\in(0,\infty)$ as
\[
p^{\nu,\tau_0(t)}(r) = \int_0^t\nu(y+r)\int_0^\infty p^\nu_s(t-y)\,ds\,dy.
\]
\end{lemma} 
\begin{proof}
This follows by performing the proof of \cite[Proposition 3.13]{DTZ18} in the simpler setting without the spatial process.\\
\end{proof}

\begin{lemma}\label{lem:densitySR} Assume \ref{H0} and \ref{H1}. Suppose  $\phi \in L^\infty((-\infty,0)\times\Omega)$ and $g\in L^\infty((0,\infty)\times\Omega)$. Then for $t>0$, $x\in \Omega$ 
\begin{align}\label{eq:3.5}
\mathbf E\left[\phi \left(t-S^\nu_{\tau_0(t)},B^x_{S^\nu_{\tau_0(t)}}\right)\mathbf1_{\{\tau_0(t)<\tau_0(\tau_\Omega(x))\}} \right]&=\,\int_0^\infty\int_\Omega\phi (-r,y)\left( p^\Omega_{t+r}(x,y)p^{\nu,\tau_0(t)}(r)\right)dy\,dr,
\intertext{and} \nonumber
\mathbf E\left[\int_0^{\tau_0(t)\wedge\tau_0(\tau_\Omega(x))}g\left(t-S^\nu_r,B^x_{S^\nu_r}\right) \,dr\right]&=  \int_\Omega \int_0^t g(t-s,y) \left(p^\Omega_s(x,y)\,\int_0^\infty p^\nu_r(s)\,dr\right)ds\,dy.
\end{align}
\end{lemma}
\begin{proof}
Extend  $\phi$ and $g$ to 0 on the appropriate cemetery state. Then, proceeding as in Remark \ref{rem:SRbry} and then using independence between $S^\nu_{\tau_0(t)}$ and $B^x_r$ along with Lemma \ref{lem:densitystau}
\begin{align*}
\mathbf E\left[\phi \left(t-S^\nu_{\tau_0(t)},B^x_{S^\nu_{\tau_0(t)}}\right)\mathbf1_{\{\tau_0(t)<\tau_0(\tau_\Omega(x))\}} \right]=&
\mathbf E\left[\phi \left(t-S^\nu_{\tau_0(t)},B^x_{S^\nu_{\tau_0(t)}}\right)\right]\\
=&\,\int_0^\infty\int_\Omega\phi (-r,y)\left( p^\Omega_{t+r}(x,y)p^{\nu,\tau_0(t)}(r)\,dy\right)dr. 
\end{align*}

The inhomogeneous term is treated similarly and we omit the computation.\\
\end{proof}

\begin{corollary}\label{cor:convergenceSR}
Assume \ref{H0} and \ref{H1}. Let $f_n,f\in L^\infty ((0,\infty)\times\Omega)$, $\phi_n,\phi\in L^\infty ((-\infty,0)\times\Omega) $, for  $n\in\mathbb N$, such that $f_n\to f$ and $\phi_n\to\phi$ bpw a.e. as $n\to\infty$.\\
Then, as $n\to\infty$ 
\[
u_n\to u \quad \text{bpw a.e. on }(-\infty,T)\times\Omega,
\]
 where $u_n$ is defined as \eqref{eq:SR} for $f\equiv f_n$, $\phi\equiv\phi_n$ on $(0,T)\times\Omega$, and as $\phi_n$ on $(-\infty,0)\times\Omega$, and $u$ is defined as \eqref{eq:SR} for $f\equiv f$, $\phi\equiv\phi$ on $(0,T)\times\Omega$, and as $\phi$ on $(-\infty,0)\times\Omega$.
\end{corollary}
\begin{proof} This is a straightforward application of DCT given Lemma \ref{lem:densitySR} and $\mathbf E[\tau_0(T)]<\infty$.
\end{proof}

\begin{example}
\begin{enumerate}[(i)]
	\item If $\phi=\phi_0$ does not depend on time, then \eqref{eq:3.5} equals
\begin{align*}
\mathbf E\left[\phi_0\left(B^x_{S^\nu_{\tau_0(t)}}\right)\mathbf1_{\{\tau_0(t)<\tau_0(\tau_\Omega(x))\}} \right] &=\int_\Omega\phi_0(y)\left(\int_0^\infty p^\Omega_{t+r}(x,y)p^{\nu,\tau_0(t)}(r)\,dr\right)dy,
\end{align*}
where $p^\Omega$ can be the density of any of the Feller processes listed in Example \ref{ex:spati}.

\item If $S^\nu=S^\alpha$, the $\alpha$-stable subordinator, $\alpha\in(0,1)$, then \cite[Formula (5.12)]{JKMS12}
\[
p^{\nu,\tau_0(t)}(r)= \int_0^t\frac{(y+r)^{-1-\alpha}}{|\Gamma(-\alpha)|}\frac{(t-y)^{\alpha-1}}{\Gamma(\alpha)} \,dy =t^{\alpha}\frac{r^{-\alpha}(t+r)^{-1}}{\Gamma(\alpha)\Gamma(1-\alpha)},
\]
and if in addition $B$ is a $d$-dimensional Brownian motion
\[
\mathbf E\left[\phi \left(t-S^\alpha_{\tau_0(t)},B^x_{S^\alpha_{\tau_0(t)}}\right) \right]= \int_0^\infty\int_{\R^d}\phi (-r,y)\left( e^{\frac{-|x-y|^2}{4(t+r)}} \frac{c_{d,\alpha}t^{\alpha}}{r^\alpha (t+r)^{d/2+1}}\right)dy\,dr,
\]
where $c_{d,\alpha}= \sin(\pi\alpha)/(2^d\pi^{d/2+1})$, so that $\mathcal L= \Delta$, the $d$-dimensional Laplacian. Moreover 
\[
\mathbf E\left[\int_0^{\tau_0(t)}g\left(t-S^\alpha_r,B^x_{S^\alpha_r}\right) \,dr\right]=  \int_{\R^d} \int_0^t g(t-s,y) \left( \frac{e^{\frac{-|x-y|^2}{4s}}}{(4\pi s)^{d/2}}\,\frac{s^{\alpha-1}}{\Gamma(\alpha)}\right)ds\,dy.
\]
\item If instead $B$ is a killed 1-$d$ Brownian motion for $\Omega=(0,\pi)$, then for $t>0$, $x\in(0,\pi)$
\begin{align*}
\mathbf E&\left[\phi \left(t-S^\alpha_{\tau_0(t)},B^x_{S^\alpha_{\tau_0(t)}}\right)\mathbf 1_{\{\tau_0(t)<\tau_0(\tau_\Omega(x))\}} \right]\\
&= \int_0^\infty\int_0^\pi\phi (-r,y)\left(\left(\sum_{n=1}^\infty e^{-n^2(t+r)}\sin(n x)\sin(n y)   \right)\frac{c_\alpha t^{\alpha}}{r^{\alpha}(t+r)}\right)dy\,dr. 
\end{align*}
where $c_\alpha = 2\sin(\alpha\pi)/\pi^2$, and $\{n^2,\sqrt{2/\pi}\sin(n\cdot)\}_{n\in\mathbb N}$ are the eigenvalues-eigenfunctions  of the Dirichlet Laplacian $\mathcal L =\Delta_\Omega$ \cite{Davies}.
\item If now $B$ is the subordination of the above killed Brownian motion by an independent $\beta$-stable L\'evy subordinator \cite{BV17,Bogdan}, so that
\[
\mathcal L u(x)=-(-\Delta_\Omega)^{\beta} u(x)= \frac{1}{|\Gamma(-\beta)|}\int_0^\infty \left(e^{r\Delta_\Omega}u(x)-u(x)\right)\frac{dr}{r^{1+\beta}},\quad \beta\in(0,1),
\]
then the homogeneous part of \eqref{eq:SR} reads, for $t>0$, $x\in(0,\pi)$,
\begin{align*}
\mathbf E&\left[\phi \left(t-S^\alpha_{\tau_0(t)},B^x_{S^\alpha_{\tau_0(t)}}\right)\mathbf 1_{\{\tau_0(t)<\tau_0(\tau_\Omega(x))\}} \right]\\
&= \int_0^\infty\int_0^\pi\phi (-r,y)\left(\left(\sum_{n=1}^\infty e^{-n^{2\beta}(t+r)}\sin(n x)\sin(n y)   \right)\frac{c_\alpha t^{\alpha}}{r^{\alpha}(t+r)}\right)dy\,dr. 
\end{align*}

\item If $B$ is the reflection at 0 of a 1-$d$ Brownian motion, then $\Omega=[0,\infty)$, $\mathcal L =\partial_x^2$ with Neumann boundary condition on $(0,T]\times\{0\}$, and for $t,x>0$
\begin{align*}
\mathbf E\left[\phi \left(B^x_{S^\alpha_{\tau_0(t)}}\right) \right]= \int_{0}^\infty\phi (y)\left(\int_0^\infty\left( e^{\frac{-|x-y|^2}{4(t+r)}}+e^{\frac{-|x+y|^2}{4(t+r)}}\right) \frac{c_{d,\alpha}t^{\alpha}}{r^\alpha (t+r)^{d/2+1}}\,dr\right)dy.
\end{align*}

\end{enumerate}
\end{example}

\section{Weak solution}\label{sec:weak}
In this section we prove that the stochastic representation \eqref{eq:SR} is a weak solution for problem \eqref{preRLHO_nl}, under the stronger assumption \ref{H1'} on the spatial semigroup $e^{\mathcal L}$. As outlined in Example \ref{ex:H10}, assumption  \ref{H1'}   applies to several operators with  Dirichlet and Neumann boundary conditions.\\
We introduce the notation 
\[
\langle f,g\rangle = \int_{-\infty}^T\int_\Omega f(t,x)g(t,x)\,dx\,dt.
\]


\begin{remark}\label{rmk:smallr}
If $f\in\text{Dom}(\mathcal L)$, then we use the symbol $\delta=\delta_f$ to denote a positive number such that  $\|e^{r\mathcal L}f-f\|_X\le r(\delta+\|\mathcal Lf\|_X)$ for all $r$ small. Then, as $\Omega$ is bounded, we can use the simple bound 
\[
\left|\int_\Omega \left(e^{r\mathcal L}f-f\right)(x)\,dx\right|\le rC,\quad\text{for all $r$ small,}
\] 
 where $C=\text{Leb}(\Omega)(\delta+\|\mathcal Lf\|_X)$.
\end{remark}

\begin{remark}Recall from Theorem \ref{thm:HO_SG}  that $\LHM$ and  $\LH$ denote abstract generators, meanwhile   $\LHMp$ and $\LHp$ denote pointwise defined formulas. 
\end{remark}
We define the adjoint operator 
\begin{align*}
H^{\nu,*} \varphi(t,x):=&\, \int_0^\infty \left(e^{r\mathcal L}\varphi(t+r,x)-\varphi(t,x)\right)\nu(r)\,dr.
\end{align*}
For our notion of weak solution we need the pairing $\langle u,H^{\nu,*}\varphi\rangle$ to be well defined for some test functions $\varphi's$ (see Definition \ref{def:weak}). Moreover, we want to allow  constant-in-time data $\phi$, so that the solution $u$ will be in $L^\infty((-\infty,T)\times\Omega)$, in general. To guarantee a well defined pairing and access dominated convergence arguments,  we now prove  that $H^{\nu,*}\varphi\in L^1((-\infty,T)\times\Omega)$.
\begin{lemma}\label{lem:phiL1}Assume \ref{H0} and \ref{H1'}.
If $\varphi=pq\in C_\infty^1(-\infty,T]\cdot \text{Dom}(\mathcal L)$ is such that $p, \partial_t p\in L^1(\R)$, then 
\[
(t,x)\mapsto\int_0^\infty \left|e^{r\mathcal L}\varphi(t+r,x)-\varphi(t,x)\right|\nu(r)\,dr \in  L^1((-\infty,T)\times\Omega),
\]
and in particular $H^{\nu,*}\varphi\in L^1((-\infty,T)\times\Omega)$.

\end{lemma}
\begin{proof}
We rewrite 
\begin{align*}
H^{\nu,*} \varphi(t,x)&=\int_0^\infty e^{r\mathcal L}q(x)\left(p(t+r)-p(t)\right)\nu(r)\,dr+p(t)\int_0^\infty \left(e^{r\mathcal L}q(x)-q(x)\right)\nu(r)\,dr\\
&=: (I+II)(t,x).
\end{align*}
Then, with inequalities holding up to a constant
\begin{align*}
\int_{\mathbb R\times\Omega} |I(t,x)|\,dx\,dt&\le \|q\|_\infty \int_{\mathbb R} \left|\int_0^\infty \left(p(t+r)-p(t)\right)\nu(r)\,dr\right|\,dt\\
&\le  \|q\|_\infty\left(\|p\|_{L^1(\R)}+\|\partial_tp\|_{L^1(\R)}\right),
\end{align*}
where we use \cite[Lemma 4.3]{DTZ18} in the second inequality. Considering $II$,
\begin{align*}
\int_{\mathbb R\times\Omega} |II(t,x)|\,dx\,dt&\le \|p\|_{L^{1}(\R)} \int_\Omega\left|\int_0^\infty  \left(e^{r\mathcal L}q(x)-q(x)\right)\nu(r)\,dr\right|\,dx\\
&\le   \|p\|_{L^{1}(\R)} \int_\Omega\int_0^\infty \left(r(\delta+\|\mathcal Lq\|_\infty) \wedge 2\|q\|_\infty\right) \nu(r)\,dr\,dx,
\end{align*}
which is finite, and we proved the claim.
\end{proof}

\begin{proposition}\label{prop:dual} Assume \ref{H0} and \ref{H1'}.
Let  $u\in L^\infty((-\infty,T]\times\Omega)$ such that $u\in \text{Dom}(\mathcal H_0)$
 if restricted to $t\ge0$.
Then for every  $\varphi\in  C_c^1(0,T)\cdot \text{Dom}(\mathcal L)$
\begin{equation}
\langle  H^\nu u, \varphi \rangle=\langle  u, H^{\nu,*} \varphi \rangle.
\label{eq:dual_pair_Halpha}
\end{equation}
\end{proposition}
\begin{proof}
Let $k>0$ such that $\varphi(t)=0$ for every $t\le k$. 
Note that we have the bound for all $t>k$
\begin{align*}
\left|e^{r\mathcal H}u(t,x)-u(t,x)\right|&= \left|e^{r\mathcal H_0}u(t,x)-u(t,x)\right|\mathbf 1_{\{r\le k\}}+\left|e^{r\mathcal L}u(t-r,x)-u(t,x)\right|\mathbf 1_{\{r> k\}}\\
&\le  r\left(\delta+\|\mathcal H_0 u\|_{C_\infty([0,T]\times	\Omega)}\right)\mathbf 1_{\{r\le k\}} +2\|u\|_{B((-\infty,T]\times	\Omega)}\mathbf 1_{\{r> k\}},
\end{align*}
 and  $\LHMp u$ (defined in \eqref{eq:LHM_c}) is well defined for each $(t,x)\in (0,T]\times\Omega$. 
 By the above remark and $\varphi\in L^1(\mathbb R\times\Omega)$ we can apply DCT in the second identity  below
\begin{align*}
\langle \LHMp u, \varphi \rangle&=\int_{-\infty}^T \int_\Omega \left(\int_0^\infty \left(e^{r\mathcal H}u(t,x)-u(t,x)\right)\nu(r)\,dr\right) \varphi(t,x)\,dx\,dt  \\
&=\lim_{\epsilon\downarrow0}\Bigg(\int_{-\infty}^T \int_\Omega \left(\int_\epsilon^\infty e^{r\mathcal H}u(t,x)\nu(r)\,dr\right) \varphi(t,x)\,dx\,dt \\
&\quad\quad\quad -\int_{-\infty}^T \int_\Omega \left(\int_\epsilon^\infty \varphi(t,x)\nu(r)\,dr\right)  u(t,x) \,dx\,dt \Bigg) \\
&=\lim_{\epsilon\downarrow0}\Bigg(\int_{-\infty}^T \int_\Omega u(t,y) \left(\int_\epsilon^\infty e^{r\mathcal L}\varphi(t+r,y)\nu(r)\,dr\right) \,dy\,dt \\
&\quad\quad\quad -\int_{-\infty}^T \int_\Omega \left(\int_\epsilon^\infty \varphi(t,y)\nu(r)\,dr\right)  u(t,y) \,dy\,dt \Bigg) \\
&=\int_{-\infty}^T \int_\Omega u(t,y) \left(\int_0^\infty \left(e^{r\mathcal L}\varphi(t+r,y)-\varphi(t,y)\right)\nu(r)\,dr\right) \,dy\,dt \\
&=\langle  u, H^{\nu,*} \varphi \rangle,
\end{align*}
where for the third identity we use \eqref{eq:sa}, Fubini's Theorem and $\varphi(t+r)= 0$ for $t\ge T-r$, and for the fourth  identity we use DCT, thanks to  Lemma \ref{lem:phiL1} and $u\in L^\infty((-\infty,T)\times\Omega)$.

\end{proof}

Our  approximation procedure, in the proof of Theorem \ref{thm:weak}, will be carried out using the following assumption on the approximating data.
\begin{description}
\item[(H2)\label{H2}] Let  $\phi$ be a linear combination of functions in $C^1_\infty(-\infty,0]\cap \{f'(0-)=0\}\cdot \text{Dom}(\mathcal L)$.
\end{description}
\begin{remark}\label{rmk:H2}
	 The  functions satisfying  \ref{H2} are dense in $L^\infty((-\infty,0)\times\Omega)$ with respect to bpw a.e. convergence. To prove it,  for $\text{Dom}(\mathcal L)\subset C_{\infty}(\Omega)$ one can use the the Stone-Weierstrass strategy in \cite[Appendix II]{T18} to show that the  functions satisfying  \ref{H2} are uniformly dense in $C_\infty((-\infty,0)\times\Omega)$, which in turn is  bpw a.e. dense in $L^\infty((-\infty,0)\times\Omega)$. If instead $\text{Dom}(\mathcal L)\subset C_{\infty}(\overline\Omega)$, then the same strategy holds, but now one should show that the  functions satisfying  \ref{H2} are uniformly dense in $C_\infty((-\infty,0)\times\overline\Omega)$.
\end{remark}

We state a natural assumption to apply Dynkin formula in the next Lemma.
\begin{description}
\item[(H2')\label{H2'}] The function  $\phi:(-\infty,0]\times \Omega\to \mathbb R$ is such that the extension of $\phi$ to $\phi(0)$ on $(0,T]\times  \Omega$ satisfies
$\phi\in \text{Dom}(\mathcal H)\subset\text{Dom}(\LHM)$.
\end{description}

\begin{remark}
	If $\phi$ satisfies \ref{H2}, then it  satisfies \ref{H2'}, as a consequence of $C^1_\infty(-\infty, T]\cdot \text{Dom}(\mathcal L)\subset \text{Dom}(\mathcal H)$. 
\end{remark}

\begin{remark}
	For the next two lemmas the domain $\Omega$ and the semigroup $e^{\mathcal L}$ only need to be as in Definition \ref{def:Hs}.
\end{remark}

\begin{lemma}\label{lem:dyn} Assume \ref{H0} and \ref{H2}. Let $g= f+f_\phi$, for $f\in L^\infty((0,T)\times\Omega)$, and
\begin{equation}
f_\phi(t,x):= \int_t^\infty \left(e^{r\mathcal H}\phi(t,x)- e^{r\mathcal L}\phi(0,x)\right) \,\nu(r)\,dr,\quad (t,x)\in(0,T]\times\Omega.
\label{eq:f_phi}
\end{equation}
Then $f_\phi$ lives in $ C_{\partial\Omega}([0,T]\times\Omega)$, and the Feynman-Kac formula   \eqref{eq:SRHO} for $g\equiv f+f_\phi,$ $ \phi_0\equiv\phi(0)$, equals the Feynman-Kac formula  \eqref{eq:SR} for $f$, $\phi$.
\end{lemma} 

\begin{proof} Extend $\phi$ to $\phi(0)$ on $(0,T]$. Observe that for $t>0$
\begin{align*}
 f_\phi(t,x)&=  \int_0^\infty \left(e^{r\mathcal H}\phi(t,x)- e^{r\mathcal L}\phi(0,x)\right) \,\nu(r)\,dr\\
&=  \int_0^\infty \left(e^{r\mathcal H}\phi(t,x)\pm \phi(0,x)- e^{r\mathcal L}\phi(0,x)\right) \,\nu(r)\,dr\\
&=  \int_0^\infty \left(e^{r\mathcal H}\phi(t,x)- \phi(0,x)\right) \,\nu(r)\,dr+\int_0^\infty \left( \phi(0,x)- e^{r\mathcal L}\phi(0,x)\right) \,\nu(r)\,dr\\
&=\LHMp\phi(t,x) -\mathcal L^{\nu}\phi(0,x),
\end{align*}
where  $\LHMp\phi\in C_{\partial\Omega}([0,T]\times\Omega)$ by \ref{H2'} and Theorem \ref{thm:HO_SG}-(i), and $\mathcal L^{\nu}\phi$ is a linear combination of elements in $ C_{\infty}(\Omega)$ by \ref{H2} and $\text{Dom}(\mathcal L)\subset\text{Dom}(\mathcal L^\nu)$.
 Rearranging, we also proved that for $t>0$ 
\begin{equation}
f_\phi+\mathcal L^\nu \phi= \LHMp\phi= \LHM\phi.
\label{eq:equal}
\end{equation}
Also, Dynkin formula \cite[Corollary of Theorem 5.1]{Dyn65}  applied to the process of Definition \ref{def:L_Omega_alpha}, gives  for $t>0$
\begin{align}\label{eq:dyn}
\mathbf E\left[\phi\left(0,B^x_{S^\nu_{\tau_0(t)}}\right)\right]-\phi(t,x)&=\, \mathbf E\left[\int_0^{\tau_0(t)} \mathcal L^{\nu} \phi\left(t-S^\nu_r,B^x_{S^\nu_r}\right)dr\right],
\end{align}
where we use $\phi(t)=\phi(0)$ on $(0,T]$ and $\phi(0)\in \text{Dom}(\mathcal L)\subset\text{Dom}(\mathcal L^{\nu})$. We conclude justifying the following equalities for $t>0$,
\begin{align*} 
\mathbf E\left[\phi\left(t-S^\nu_{\tau_0(t)},B^x_{S^\nu_{\tau_0(t)}}\right)\right]&=\, \mathbf E\left[\int_0^{\tau_0(t)} \mathcal H^\nu \phi\left(t-S^\nu_r,B^x_{S^\nu_r}\right)dr\right]+\phi(t,x)\\
&=\, \mathbf E\left[\int_0^{\tau_0(t)} (f_\phi+\mathcal L^{\nu}\phi)\left(t-S^\nu_r,B^x_{S^\nu_r}\right)dr\right]+\phi(t,x)\\
&=\, \mathbf E\left[\int_0^{\tau_0(t)} f_\phi\left(t-S^\nu_r,B^x_{S^\nu_r}\right)dr\right]+\phi(t,x)\\
&\quad\quad+\mathbf E\left[\phi\left(0,B^x_{S^\nu_{\tau_0(t)}}\right)\right]-\phi(t,x).
\end{align*}
The first equality holds by  Dynkin formula \cite[Corollary of Theorem 5.1]{Dyn65}  combining Theorem \ref{thm:HO_SG}-(i) and \ref{H2'}; the second equality holds by \eqref{eq:equal}; the third equality holds by \eqref{eq:dyn}. \\
\end{proof}

We will also use the following  lemma.
\begin{lemma}\label{lem:cap_to_mar} Assume \ref{H0}.
If $ u\in \text{Dom}(\mathcal H_0)$  and it is extended to $\phi\in L^\infty((-\infty,0)\times\Omega)$  for $t<0$, then $ H^\nu \tilde u =\LHp u  +f_\phi$ for $t>0$, where $f_\phi$ is defined as in \eqref{eq:f_phi} and $\tilde u$ is the extension of $ u$ to $\phi$ on $t\ge 0$.
\end{lemma}
\begin{proof}
Exploiting \eqref{eq:LHM_c}, simply compute for $t>0$, $x\in\Omega$ 
\begin{align*}
	\LHMp \tilde u(t,x)&= \int_0^\infty \left(e^{s\mathcal H}\tilde u(t,x)-\tilde u(t,x)\right) \,\nu(r)\,dr\\
&= \int_0^t \left( e^{r\mathcal H}u(t,x)-u(t,x)\right) \,\nu(r)\,dr+\int_t^\infty \left( e^{r\mathcal H} \phi(t,x)-u(t,x)\right) \,\nu(r)\,dr\\
&\quad\quad \pm \int_t^\infty e^{r\mathcal L} \phi(0,x) \,\nu(r)\,dr\\
&= \int_0^t \left( e^{r\mathcal H}u(t,x)-u(t,x)\right) \,\nu(r)\,dr+\int_t^\infty \left( e^{r\mathcal L} \phi(t,x)-u(t,x)\right) \,\nu(r)\,dr\\
&\quad\quad  +\int_t^\infty \left(e^{r\mathcal H} \phi(t,x)-e^{r\mathcal L} \phi(t,x)\right) \,\nu(r)\,dr \\
&=\LHp u(t,x)  +f_\phi(t,x).
\end{align*}
\end{proof}

We now define our weak  solution  for problem \eqref{preRLHO_nl}.

\begin{definition}\label{def:weak}
For given $f\in L^\infty((0,T)\times\Omega)$ and $\phi\in L^\infty((-\infty,0)\times\Omega)$, a function $u$ is said to be a \emph{weak solution to } \eqref{preRLHO_nl} if $u\in L^\infty((-\infty,T)\times\Omega)$ and 
\begin{equation}
\left\{\begin{split}
\langle u,H^{\nu,*}\varphi\rangle&=\langle -f,\varphi\rangle, &&\text{for } \varphi\in  C_c^1(0,T)\cdot \text{Dom}(\mathcal L), \\ 
u&=\phi, &&\text{a.e. on } (-\infty,0)\times\Omega.
\end{split}
\right.
\label{eq:weak_sol}
\end{equation}
\end{definition}

\color{red}
\color{black}
\begin{theorem}\label{thm:weak}
Assume \ref{H0} and \ref{H1'}, and let $f\in L^\infty((0,T)\times\Omega)$, $\phi\in L^\infty((-\infty,0)\times\Omega)$. Then the Feynman-Kac formula defined in \eqref{eq:SR} is a weak solution to \eqref{preRLHO_nl}.
\end{theorem}
\begin{proof} We assume for the first two steps that $\phi$ satisfies \ref{H2}. The proof for $e^{\mathcal L}$ acting on $C_\infty(\overline\Omega)$ is essentially identical\footnote{The only differences are in Step 1, where the Banach space for $\mathcal H^\nu_0$ is $C_\infty([0,T]\times\overline\Omega)$, and in Step 2, where the sequence $\{f_n\}_{n\in\mathbb N}$ will have to be selected from $C_\infty([0,T]\times\overline\Omega)$.}, and we omit it.\\

\emph{Step 1)} Let $u\in \text{Dom}(\mathcal H^\nu_0)$ be the unique solution in the domain of the generator to problem \eqref{postRLHO} for $g\equiv f+f_\phi$ and $\phi_0\equiv \phi(0)$, where $f_\phi\in C_{\partial\Omega}([0,T]\times\Omega)$ by Lemma \ref{lem:dyn}, and some $f\in C_{\partial\Omega}([0,T]\times\Omega)$ such that $f(0)=-f_\phi(0)-\mathcal L^{\nu}\phi(0)$. This implies that for any $\varphi\in  C_c^1(0,T)\cdot \text{Dom}(\mathcal L)$
\begin{equation}
\langle \LH  u+f_\phi, \varphi \rangle=\langle -f, \varphi \rangle.
\label{eq:sdg_au}
\end{equation}
By Theorem \ref{thm:HO_SG}-(iv) we are guaranteed that $u-\phi(0)\in \text{Dom}(\mathcal H^{\nu,\text{kill}}_0)$. Then, by Theorem \ref{thm:HO_SG}-(iii), we can pick $\{\hat u_n\}_{n\ge1}\subset \text{Dom}(\mathcal H_0^{\text{kill}})$ such that $\hat u_n\to u-\phi(0)$ and 
\[
 \LHp \hat u_n= \LH \hat u_n=\LHkill \hat u_n\to \LHkill( u-\phi(0))=\LH( u-\phi(0)),
\]
both uniformly as $n\to\infty$. Then, $u_n:=\hat u_n+\phi(0)\to u$ uniformly with $u_n(0)=\phi(0) \text{ for all }n$, and 
\begin{equation}
 \LHp  u_n = \LH \hat u_n+ \LH \phi(0)\to \LH( u-\phi(0))+\LH\phi(0) = \LH u,
\label{eq:conver}
\end{equation}
with uniform convergence, where we use Lemma \ref{lem:phi_0} and the linearity of $\LH$.
Define the extension of $u$ as 
\begin{equation}
\tilde u:=
\left\{\begin{split}
&u,&t>0,\\
&\phi,& t\le0.
\end{split}
\right.
\label{eq:extension}
\end{equation}
Then, for every $\varphi\in  C_c^1(0,T)\cdot \text{Dom}(\mathcal L)$, we can apply DCT as $n\to\infty$ to obtain   
\begin{align*}
\langle -f, \varphi \rangle\leftarrow\, \langle \LHp  u_n+f_\phi, \varphi \rangle=\langle \LHMp  \tilde u_n, \varphi \rangle= \langle  \tilde u_n,  H^{\nu,*}\varphi \rangle\to \langle  \tilde u,  H^{\nu,*}\varphi \rangle,
\end{align*}
where we use \eqref{eq:conver} and \eqref{eq:sdg_au}  in the first convergence, Lemma \ref{lem:cap_to_mar} with $ u_n\in \text{Dom}(\mathcal H_0)$ in the first equality,
 Proposition \ref{prop:dual} in the second equality with \ref{H2} and $ u_n\in \text{Dom}(\mathcal H_0)$, and Lemma \ref{lem:phiL1} with $\tilde u_n\to\tilde u$ uniformly on $(-\infty,T]\times\Omega$ for the second convergence, 
 where $\tilde u_n,$ $\tilde u$ are respectively  the extensions of $u_n$, $u$ to $\phi$ as defined in \eqref{eq:extension}. \\

\emph{Step 2)} For $f\in L^\infty((0,T)\times\Omega)$, let $u$ be the generalised solution to problem \eqref{postRLHO} for $g\equiv f+f_\phi$ and $\phi_0\equiv\phi(0)$.  Now pick a sequence  $\{f_n\}_{n\ge 1}\subset C_{\partial\Omega}([0,T]\times\Omega)$ such that $f_n\to f$ bpw a.e., and $f_n(0)=-f_\phi(0)-\mathcal L^{\nu}\phi(0)$ for each $n\in\mathbb N$. Then the respective solutions in the domain of the generator $u_n$ converge bpw to $u$, by Remark \ref{rem:convergenceSR}. And so for every $\varphi\in C_c^1(0,T)\cdot\text{Dom}(\mathcal L)$
\begin{align*}
\langle -f, \varphi \rangle\leftarrow \langle -f_n, \varphi \rangle= \langle  \tilde u_n,  H^{\nu,*}\varphi \rangle\to \langle  \tilde u,  H^{\nu,*}\varphi \rangle,
\end{align*}
where we can apply DCT in the second convergence thanks to  Lemma \ref{lem:phiL1}, and the equality holds by Step 1, where again the functions are extended to $\phi$ as in \eqref{eq:extension}.\\

\emph{Step 3)} Let $\phi\in L^\infty((-\infty,0)\times\Omega)$ and $f\in L^\infty((0,T)\times\Omega)$ and denote by $u$ the Feynman-Kac formula defined in \eqref{eq:SR} for such $\phi$ and $f$ and $t>0$, and denote by $\tilde u$ the extension of $u$ to $\phi$ for $t<0$. By Remark \ref{rmk:H2} we can take $\phi_n\to \phi$ bpw a.e., and $\phi_n$ satisfies \ref{H2} for each $n\in\mathbb N$. Denote by $\tilde u_n$ the extension of $u_n$ to $\phi_n$ as in \eqref{eq:extension}, where $u_n$ is the generalised solution to problem \eqref{postRLHO} for $g\equiv f+f_{\phi_n}$ and $\phi_0\equiv\phi_n(0)$. Then, by Lemma \ref{lem:dyn} combined with the representation \eqref{eq:SRHO} of each $u_n$, we can apply  Corollary \ref{cor:convergenceSR} to obtain  as $n\to\infty$
\[
\tilde u_n\to \tilde u\text{ bpw a.e. on }(-\infty,T]\times\Omega.
\]
 Then, for every $\varphi\in C_c^1(0,T)\cdot \text{Dom}(\mathcal L)$,
\[
\langle -f, \varphi \rangle=\langle \tilde u_n, H^{\nu,*}\varphi \rangle\to \langle \tilde u, H^{\nu,*}\varphi \rangle,\quad \text{as }n\to\infty,
\]
where we use Step 2 for the equality and we use Lemma \ref{lem:phiL1} to apply DCT, and we are done.

\end{proof}



\section*{Acknowledgements}
J. Herman and  I. Johnston are supported by the UK EPSRC funding as a part of the MASDOC DTC, Grant reference number EP/HO23364/1. 

\newpage
\let\oldbibliography\thebibliography
\renewcommand{\thebibliography}[1]{\oldbibliography{#1}
\setlength{\itemsep}{0pt}}

\end{document}